\documentclass[leqno,12pt]{article} 
\usepackage{amsmath, amssymb}
\usepackage{amsthm} 
\usepackage{pictexwd,dcpic}
\usepackage{xfrac}
\usepackage{xypic}

\theoremstyle{plain} 
\newtheorem{theorem}{\indent\sc Theorem}[section]
\newtheorem{lemma}[theorem]{\indent\sc Lemma}
\newtheorem{corollary}[theorem]{\indent\sc Corollary}
\newtheorem{proposition}[theorem]{\indent\sc Proposition}

\newtheorem{conjecture}[theorem]{\indent\sc Conjecture}

\theoremstyle{definition} 

\newtheorem{remark}[theorem]{\indent\sc Remark}
\newtheorem{example}[theorem]{\indent\sc Example}
\newtheorem{notation}[theorem]{\indent\sc Notation}

\numberwithin{equation}{section}

\def\O{\omega}

\def\P{\psi}

\def\o{\mathcal{O}}
\def\F{\mathfrak{f}}

\makeatletter
\def\address#1#2{\begingroup
\noindent\parbox[t]{7.8cm}{%
\small{\scshape\ignorespaces#1}\par\vskip1ex
\noindent\small{\itshape E-mail address}%
\/: #2\par\vskip4ex}\hfill%
\endgroup}%

\makeatother
\title{{On the Schertz conjecture}} 
\author{
\textsc{Ja Kyung Koo and Dong Sung Yoon$^*$} 
}

\date{} 
\begin{document}

\maketitle

\footnote{ 
2010 \textit{Mathematics Subject Classification}. 11R37 (primary), 11G16  (secondary). }
\footnote{ 
\textit{Key words and phrases}. class field theory, modular units} 
\footnote{$^*$Corresponding author.}
\footnote{\thanks{
The second (corresponding) author was
supported by Basic Science Research Program through the National Research Foundation of
Korea(NRF) funded by the Ministry of Education (2017R1D1A1B03030015).
}}


\begin{abstract}
Schertz conjectured that every finite abelian extension of imaginary quadratic fields can be generated by 
the norm of the Siegel-Ramachandra invariants.
We shall present a conditional proof of his conjecture by means of the characters on class groups and the  second Kronecker limit formula.

\end{abstract}

\maketitle

\section{Introduction}

Let $K$ be an imaginary quadratic field, $\mathfrak{f}$ be a nonzero integral ideal of $K$ and 
$\mathrm{Cl}(\F)$ be the ray class group of $K$ modulo $\mathfrak{f}$.
Then there exists a unique abelian extension $K_\mathfrak{f}$ of $K$ whose Galois group is isomorphic to $\mathrm{Cl}(\F)$ via the Artin map
\begin{equation}\label{artin map}
\sigma_\mathfrak{f}:\mathrm{Cl}(\F)\xrightarrow{~\sim~} \mathrm{Gal}(K_{\F}/K),
\end{equation}
which is called the \textit{ray class field} of $K$ modulo $\mathfrak{f}$.
By class field theory any abelian extension of $K$ is contained in some ray class field $K_\mathfrak{f}$, and hence it is important to construct the ray class fields of $K$ to figure out the maximal abelian extension of $K$.
\par

In 1964, Ramachandra (\cite[Theorem 10]{Ramachandra}) constructed a primitive generator of $K_\F$ over $K$ in terms of certain elliptic unit and showed that arbitrary finite abelian extension of $K$ can be generated by the norm of this unit, which settled down the Kronecker's Jugendtraum over an imaginary quadratic field.
    However, his unit involves too complicated products of singular values of the Klein forms and the discriminant $\Delta$-function to use in practice.
On the other hand, Schertz (\cite[Theorem 6.8.4]{Schertz}) presented a relatively simple ray class invariant over $K$  by means of the singular value of certain Siegel function, namely, Siegel-Ramachandra invariant.
He further conjectured that every finite abelian extension of $K$ can be generated by 
the norm of the Siegel-Ramachandra invariant (\cite[Conjecture 6.8.3]{Schertz}) as follows:

\begin{conjecture}
Let $\F$ be a nonzero proper integral ideal of $K$ and let $L$ be a finite abelian extension of $K$ such that $K\subset L\subset K_{\F}$.
Then for every nonzero integer $n$ and $C\in\mathrm{Cl}(\mathfrak{f})$
\begin{equation*}
L=K(N_{K_{\F}/L}(g_\F(C)^{n})),
\end{equation*}
where $g_\F(C)$ is the Siegel-Ramachandra invariant of conductor $\F$ at $C$ defined in \textup{(\ref{invariant})}.
\end{conjecture}

Recently, Koo-Yoon generated ray class fields $K_\F$ over $K$ via Siegel-Ramachandra invariants by making use of the characters on class groups and the  second Kronecker limit formula (\cite[Theorem 4.6]{K-Y}).
In this paper by improving their idea we shall give a conditional proof of the conjecture with certain assumption 
depending only on the extension degree $[K_{\F}:LH_K]$, where $H_K$ denotes the Hilbert class field of $K$ (Theorem \ref{singular value} and Example \ref{example}).

\begin{notation}
For $z\in\mathbb{C}$, we denote by $\overline{z}$ the complex conjugate of $z$.
If $G$ is a group and $g_1,g_2,\ldots,g_r$ are elements of $G$, let $\langle g_1,g_2,\ldots,g_r \rangle$ be the subgroup of $G$ generated by $g_1,g_2,\ldots,g_r$.
Moreover, if $H$ is a subgroup of $G$ and $g\in G$, by $[g]$ we mean  the coset $gH$ of $H$ in $G$.
For a number field $K$, let $\o_K$ be the ring of integers of $K$.
If $a\in\o_K$, we denote by $(a)$ the principal ideal of $K$ generated by $a$.
\end{notation}

\section{Main Theorem}\label{Main Theorem}

For a rational vector $\mathbf{r}=\left[\begin{matrix}r_1\\r_2\end{matrix}\right]\in\mathbb{Q}^2\setminus\mathbb{Z}^2$, we define the \textit{Siegel function} $g_{\mathbf{r}}(\tau)$ on the complex upper half plane $\mathbb{H}$ by the following infinite product
\begin{equation*}
g_{\mathbf{r}}(\tau)=-q^{\frac{1}{2}\mathbf{B}_2(r_1)}e^{\pi \mathrm{i} r_2(r_1-1)}(1-q^{r_1}e^{2\pi \mathrm{i} r_2})\prod_{n=1}^\infty(1-q^{n+r_1}e^{2\pi \mathrm{i} r_2})(1-q^{n-r_1}e^{-2\pi \mathrm{i} r_2}),
\end{equation*}
where $\mathbf{B}_2(X)=X^2-X+{1}/{6}$ is the second Bernoulli polynomial and $q=e^{2\pi \mathrm{i}\tau}$.
Then a Siegel function is a modular unit, namely, it is a modular function whose zeros and poles are supported     only at the cusps (\cite{Siegel} or \cite[p.36]{Kubert}).
In particular, if $\mathbf{r}\in({1}/{N})\mathbb{Z}^2\setminus\mathbb{Z}^2$ with an integer $N\geq 2$ then
the function $g_{\mathbf{r}}(\tau)^{12N}$ belongs to $\mathcal{F}_N$ (\cite[Proposition 1.1]{Jung}), where $\mathcal{F}_N$ is the field of meromorphic modular functions for the principal congruence subgroup $\Gamma(N)$  whose Fourier coefficients lie in the $N$th cyclotomic field $\mathbb{Q}(e^{2\pi \mathrm{i}/N})$.
\par

Let $K$ be an imaginary quadratic field of discriminant $d_K$,
$\mathfrak{f}$ be  a nonzero proper integral ideal of $K$ and $N$ be the smallest positive integer in $\mathfrak{f}$.
For $C\in\mathrm{Cl}(\F)$, we take any integral ideal $\mathfrak{c}$ in $C$ and choose a basis $[\omega_1,\omega_2]$ of $\mathfrak{f}\mathfrak{c}^{-1}$ such that ${\omega_1}/{\omega_2}\in\mathbb{H}$.
Then one can write
\begin{equation*}
N=r_1\omega_1+r_2\omega_2
\end{equation*}
for some $r_1,r_2\in\mathbb{Z}$.
We define the \textit{Siegel-Ramachandra invariant} of conductor $\F$ at $C$  by 
\begin{equation}\label{invariant}
g_\F(C)=g_{\left[\begin{smallmatrix}r_1/N\\r_2/N\end{smallmatrix}\right]}({\omega_1}/{\omega_2})^{12N}.
\end{equation}
This value depends only on the class $C$ and $\F$, not on the choice of $\mathfrak{c}$.

\begin{proposition}\label{Galois action}
Let $C, C'\in\mathrm{Cl}(\F)$ with $\F\neq \o_K$.
\begin{itemize}
\item[\textup{(i)}] $g_\F(C)$ belongs to $K_\F$ as an algebraic integer. 
If $N$ is composite, $g_\F(C)$ is a unit in $K_\F$.
\item[\textup{(ii)}]
We have the transformation formula
\begin{equation*}
g_\F(C)^{\sigma_\F(C')}=g_\F(CC'),
\end{equation*}
where $\sigma_\F$ is the Artin map stated in \textup{(\ref{artin map})}.
\end{itemize}
\end{proposition}
\begin{proof}
\cite[Chapter 19, Theorem 3]{Lang} and \cite[Chapter 11, Theorem 1.2]{Kubert}.
\end{proof}

Let $\chi$ be a nontrivial character of $\mathrm{Cl}(\F)$ with $\F\neq\o_K$, $\F_\chi$ be a conductor of $\chi$ and
$\chi_0$ be the primitive character of $\mathrm{Cl}(\F_\chi)$ corresponding to $\chi$.
The \textit{Stickelberger element} and the \textit{L-function} for $\chi$ are defined by
\begin{eqnarray*}
S_\F(\chi)&=&\sum_{C\in\mathrm{Cl}(\F)}\chi(C)\log|g_\F(C)|,\\
L_\F(s,\chi)&=&\sum_{\substack{(0)\neq \mathfrak{a}\subset \o_K \\ \gcd(\mathfrak{a},\F)=1 }} \frac{\chi(\mathfrak{a})}{\mathcal{N}(\mathfrak{a})^s}\quad (s\in\mathbb{C}),
\end{eqnarray*}
respectively, where $\mathcal{N}(\mathfrak{a})$ is the absolute norm of an ideal $\mathfrak{a}$.
The second Kronecker limit formula describes the relation between the Stickelberger element and the L-function as follows:

\begin{proposition}\label{L-function relation}
Let $\chi$ be a nontrivial character of $\mathrm{Cl}(\F)$ with $\F_\chi\neq\o_K$.
Then we have
\begin{equation*}
L_{\F_\chi}(1,\chi_0)\prod_{\substack{\mathfrak{p}\,|\,\F \\ \mathfrak{p}\,\nmid\,\F_\chi }}(1-\overline{\chi_0}([\mathfrak{p}]))
=-\frac{2\pi \chi_0([\gamma\mathfrak{d}_K \F_\chi])}{6N(\F_\chi) \O(\F_\chi)T_\gamma(\overline{\chi_0})\sqrt{-d_K}}  \cdot S_\F(\overline{\chi}),
\end{equation*}
where $\mathfrak{d}_K$  is the different  of $K/\mathbb{Q}$, $\gamma$ is an element of $K$ such that $\gamma\mathfrak{d}_K\F_\chi$ is an integral ideal of $K$ prime to $\F_\chi$,
$N(\F_\chi)$ is the smallest positive integer in $\F_\chi$,
$\O(\F_\chi)$ is the number of roots of unity in $K$ which are $\equiv 1\pmod{\F_\chi}$ and
\begin{equation*}
T_\gamma(\overline{\chi_0})=\sum_{x+\F_\chi\in(\o_K/\F_\chi)^\times}\overline{\chi_0}([x\o_K])e^{2\pi \mathrm{i} \mathrm{Tr}_{K/\mathbb{Q}}(\gamma x)}.
\end{equation*}

\end{proposition}
\begin{proof}
\cite[Chapter 11 \S2, LF 2]{Kubert}.
\end{proof}

\begin{remark}\label{Stickremark}
 Since $\chi_0$ is a nontrivial primitive character of $\mathrm{Cl}(\mathfrak{f}_\chi)$, both $L_{\mathfrak{f}_\chi}(1,\chi_0)$ and the Gauss sum  $T_\gamma(\overline{\chi}_0)$ are nonzero (\cite[Chapter V, Theorem 10.2]{Janusz}, \cite[Chapter 22 $\S$1, G 3]{Lang}).
If every prime ideal factor of $\mathfrak{f}$ divides $\mathfrak{f}_\chi$ then we understand the Euler factor 
$\prod_{{\mathfrak{p}\,|\,\F,~  \mathfrak{p}\,\nmid\,\F_\chi }}(1-\overline{\chi_0}([\P]))$ to be $1$, and hence
we conclude $S_\F(\overline{\chi})\neq 0$.
\end{remark}

For an intermediate field $L$ of the extension $K_\F/K$, we denote by $\mathrm{Cl}(K_\F/L)$ the subgroup of $\mathrm{Cl}(\F)$ corresponding to $\mathrm{Gal}(K_\F/L)$ via the Artin map (\ref{artin map}).
Then one can identify $\mathrm{Cl}(K_\F/H_K)$ with the quotient group
\begin{equation*}
(\o_K/\mathfrak{f})^\times  / \{\alpha+\mathfrak{f}\in (\o_K/\mathfrak{f})^\times~|~\alpha\in \o_K^\times \}
\end{equation*}
via the natural homomorphism
\begin{equation*}
\begin{array}{ccc}
(\o_K /\mathfrak{f})^\times&\longrightarrow &\mathrm{Cl}(K_\F/H_K)\\
\alpha +\F&\longmapsto& [(\alpha)].
\end{array}
\end{equation*}
Let $\mathfrak{f}=\prod_\mathfrak{p}\mathfrak{p}^{e_\mathfrak{p}}$ be a prime ideal factorization of $\F$.
For each prime ideal $\mathfrak{p}$, we set
\begin{equation*}
\mathbf{G}_\mathfrak{p}=(\o_K/\mathfrak{p}^{e_\mathfrak{p}})^\times  / \{\alpha+\mathfrak{p}^{e_\mathfrak{p}}\in (\o_K/\mathfrak{p}^{e_\mathfrak{p}})^\times~|~\alpha\in \o_K^\times \}
\end{equation*}
so that $\mathbf{G}_\mathfrak{p}\cong \mathrm{Cl}(K_{\mathfrak{p}^{e_\mathfrak{p}}}/H_K)\subset\mathrm{Cl}(\mathfrak{p}^{e_\mathfrak{p}})$.
Then we have 
\begin{equation*}
|\mathbf{G}_\mathfrak{p}|=\phi(\mathfrak{p}^{e_\mathfrak{p}})\frac{\omega(\mathfrak{p}^{e_\mathfrak{p}})}{\omega_K}
\end{equation*}
where $\phi(\mathfrak{p}^{e_\mathfrak{p}})=|(\mathcal{O}_K/\mathfrak{p}^{e_\mathfrak{p}})^\times|$, $\omega_K$ is the number of roots of unity in $K$, and $\omega(\mathfrak{p}^{e_\mathfrak{p}})$ is the number of roots of unity in $K$ which are $\equiv 1\pmod{\mathfrak{p}^{e_\mathfrak{p}}}$.

\begin{lemma}\label{extension}
Let $H\subset G$ be two finite abelian groups, $g\in G\setminus H$ and $n$ be the order of the coset $[g]$ in $G/H$.
Then for any character $\chi$ of $H$, we can extend it to a character $\psi$ of $G$ in such a way that $\psi(g)$ is any fixed $n$-th root of $\chi(g^n)$.

\end{lemma}
\begin{proof}
\cite[Chapter VI, Proposition 1]{Serre}.
\end{proof}

Let $L$ be a finite abelian extension of $K$ such that $K\subsetneq L\subset K_{\F}$ and $L\not\subset H_K$.
Replacing $\mathfrak{f}$ by $\mathfrak{f}\mathfrak{p}^{-e_\mathfrak{p}}$ if necessary, we may assume that $L\not\subset K_{\mathfrak{f}\mathfrak{p}^{-e_\mathfrak{p}}}$ for every prime ideal factor $\mathfrak{p}$ of $\mathfrak{f}$.

\begin{lemma}\label{existence of character}
Assume that for each prime ideal factor $\mathfrak{p}$ of $\mathfrak{f}$ there is a rational prime $\nu_\mathfrak{p}$ satisfying 
$\mathrm{ord}_{\nu_\mathfrak{p}}(|\mathbf{G}_\mathfrak{p}|)>\mathrm{ord}_{\nu_\mathfrak{p}}([K_{\F}:LH_K])+i_\mathfrak{p}$
where 
\begin{equation*}
i_\mathfrak{p}=\left\{
\begin{array}{ll}
0&\textrm{if $\nu_\mathfrak{p}\neq 2$}\\
1&\textrm{if $\nu_\mathfrak{p}=2$}.
\end{array}\right.
\end{equation*}
Then, for any class $D\in \mathrm{Cl}(\F)\setminus \mathrm{Cl}(K_\F/L)$ there exists a character $\chi$ of $\mathrm{Cl}(\F)$ such that $\chi|_{\mathrm{Cl}(K_\F/L)}=1$, $\chi(D)\neq 1$ and $\mathfrak{p}\,|\,\F_\chi$ for every prime ideal factor $\mathfrak{p}$ of $\mathfrak{f}$.
\end{lemma}

\begin{proof}
By Lemma \ref{extension}, there exists a character $\chi$ of $\mathrm{Cl}(\F)$ satisfying $\chi|_{\mathrm{Cl}(K_\F/L)}=1$ and $\chi(D)\neq 1$.
For each $\mathfrak{p}$, we define a homomorphism $\varphi_\mathfrak{p}$ by
\begin{equation*}
\begin{array}{cccc}
\varphi_\mathfrak{p}:&\mathrm{Cl}(K_\F/H_K)&\rightarrow& \mathbf{G}_\mathfrak{p}
\\
&[\alpha+\F]&\longmapsto&[\alpha+\mathfrak{p}^{e_\mathfrak{p}}].
\end{array}
\end{equation*}
Suppose that $\mathfrak{p}\nmid \F_\chi$ for some $\mathfrak{p}$.
Let $n$ be the order of the class $D$ in the quotient group $\mathrm{Cl}(\F)/\mathrm{Cl}(K_\F/H_K)$.
Then $D^n=[(\beta)]$ for some $\beta\in\mathcal{O}_K$ which is prime to $\F$.
Note that
\begin{equation*}
\mathrm{Cl}(K_\F/L)\cap\mathrm{Cl}(K_\F/H_K)=\mathrm{Cl}(K_\F/L H_K).
\end{equation*}
\begin{enumerate}
\item[\textbf{Case 1.}]
First, suppose that $\mathbf{G}_\mathfrak{p}/\mathrm{Im}(\varphi_\mathfrak{p}|_{\mathrm{Cl}(K_\F/L H_K)})\neq\langle [\beta+\mathfrak{p}^{e_\mathfrak{p}}]\rangle$.
Then there exists a nontrivial character $\psi$ of $\mathbf{G}_\mathfrak{p}$ in such a way that  $\psi$ is trivial on $\mathrm{Im}(\varphi_\mathfrak{p}|_{\mathrm{Cl}(K_\F/L H_K)})$ and $\psi([\beta+\mathfrak{p}^{e_\mathfrak{p}}])=1$.
Let $\psi'=\psi \circ \varphi_\mathfrak{p}$ be a character of $\mathrm{Cl}(K_\F/H_K)$.
Then it is possible for us to extend $\psi'$ to a character $\psi_\mathfrak{p}$ of $\mathrm{Cl}(\F)$ such that $\psi_\mathfrak{p}|_{\mathrm{Cl}(K_\F/L)}=1$ and $\psi_\mathfrak{p}(D)=1$  by Lemma \ref{extension}.

\item[\textbf{Case 2.}] 
Now, assume that $\mathbf{G}_\mathfrak{p}/\mathrm{Im}(\varphi_\mathfrak{p}|_{\mathrm{Cl}(K_\F/L H_K)})=\langle [\beta+\mathfrak{p}^{e_\mathfrak{p}}]\rangle$.
By the hypothesis, there is a nontrivial character $\psi$ of $\mathbf{G}_\mathfrak{p}$ such that 
$\psi$ is trivial on $\mathrm{Im}(\varphi_\mathfrak{p}|_{\mathrm{Cl}(K_\F/L H_K)})$ and $\psi([\beta+\mathfrak{p}^{e_\mathfrak{p}}])\neq 1, \chi(D^n)^{-1}$.
Similarly as in \textbf{Case 1} one can extend $\psi$ to a character $\psi_\mathfrak{p}$ of $\mathrm{Cl}(\F)$ for which $\psi_\mathfrak{p}|_{\mathrm{Cl}(K_\F/L)}=1$ and $\psi_\mathfrak{p}(D)\neq \chi(D)^{-1}$.

\end{enumerate}

Here we observe that $\psi_\mathfrak{p}$ is a nontrivial character whose conductor is solely divisible by $\mathfrak{p}$ in both cases.
Hence the character $\chi\psi_\mathfrak{p}$ of $\mathrm{Cl}(\F)$ satisfies $\chi\psi_\mathfrak{p}|_{\mathrm{Cl}(K_\F/L)}=1$, $\chi\psi_\mathfrak{p}(D)\neq 1$, $\mathfrak{p}\,|\,\F_{\chi\psi_\mathfrak{p}}$ and $\F_\chi\,|\,\F_{\chi\psi_\mathfrak{p}}$.
Thus, we replace $\chi$ by $\chi\psi_\mathfrak{p}$.
By continuing this process for every $\mathfrak{p}$, we get the lemma
.

\end{proof}

Let $\mathbf{h}_{L,\F}$ be the set of prime ideal factors $\mathfrak{p}$ of $\F$ such that there is no rational prime $\nu_\mathfrak{p}$ satisfying 
$\mathrm{ord}_{\nu_\mathfrak{p}}(|\mathbf{G}_\mathfrak{p}|)>\mathrm{ord}_{\nu_\mathfrak{p}}([K_{\F}:LH_K])+i_\mathfrak{p}$.

\begin{theorem}\label{singular value}
Let $\mathfrak{f}=\prod_\mathfrak{p}\mathfrak{p}^{e_\mathfrak{p}}$ be  a nonzero proper integral ideal of $K$ and $L$ be a finite abelian extension of $K$ such that $K\subset L\subset K_{\F}$. 
Assume that $L\not\subset H_K$ and $L\not\subset K_{\mathfrak{f}\mathfrak{p}^{-e_\mathfrak{p}}}$ for every prime ideal factor $\mathfrak{p}$ of $\mathfrak{f}$, and 
\begin{equation}\label{assumption}
\sum_{\mathfrak{p}\in \mathbf{h}_{L,\F}}\frac{1}{[L:L\cap K_{\F\mathfrak{p}^{-e_\mathfrak{p}}}]}\leq\frac{1}{2}.
\end{equation}
Then, for any nonzero integer $n$ and $C\in\mathrm{Cl}(\mathfrak{f})$ the singular value 
\begin{equation*}
N_{K_{\F}/L}(g_\F(C)^{n})
\end{equation*}
generates $L$ over $K$.
In particular, if $| \mathbf{h}_{L,\F}|=0$ or $1$, then the assumption $(\ref{assumption})$ is always true and so we have the desired result.
\end{theorem}

\begin{proof}
It is clear when $L=K$, and so we may assume that $K\subsetneq L$.
Let 
\begin{equation*}
L'=K(N_{K_{\F}/L}(g_\F(C_0)^{n})),
\end{equation*}
where $C_0$ is the unit class in $\mathrm{Cl}(\mathfrak{f})$.
On the contrary, suppose $L'\subsetneq L$.
Then we claim that there exists a character $\chi$ of $\mathrm{Cl}(\F)$ satisfying $\chi|_{\mathrm{Cl}(K_\F/L)}=1$, $\chi|_{\mathrm{Cl}(K_\F/L')}\neq 1$ and $\mathfrak{p}\,|\,\F_\chi$ for every $\mathfrak{p}\in \mathbf{h}_{L,\F}$.
Indeed, if $| \mathbf{h}_{L,\F}|=0$ then it is clear by Lemma \ref{extension}.
Suppose $| \mathbf{h}_{L,\F}|\geq 1$. 
Let
\begin{eqnarray*}
G_1
&=&\left\{\textrm{characters $\chi$ of $\mathrm{Cl}(\F)~|~ \chi|_{\mathrm{Cl}(K_\F/L)}=1, \chi|_{\mathrm{Cl}(K_\F/L')}\neq 1$} \right\},\\
G_2
&=&\left\{\textrm{nontrivial characters $\chi$ of $\mathrm{Cl}(\F)~|~ \chi|_{\mathrm{Cl}(K_\F/L)}=1$ and $ \mathfrak{p}\nmid\F_{\chi}$ for some $\mathfrak{p}\in \mathbf{h}_{L,\F}$} \right\}
\end{eqnarray*}
Observe that all characters in $G_1$ are nontrivial.
Then we have
\begin{eqnarray*}
|G_1|
&=&\left|\left\{\textrm{characters $\chi$ of $\mathrm{Gal}(L/K)~|~ \chi|_{\mathrm{Gal}(L/L')}\neq 1$} \right\}\right|\quad\textrm{since  $\mathrm{Cl}(\F)/\mathrm{Cl}(K_\F/L)\cong\mathrm{Gal}(L/K)$}\\
&=&\left|\left\{\textrm{characters $\chi$ of $\mathrm{Gal}(L/K)$} \right\}\right|
-\left|\left\{\textrm{characters $\chi$ of $\mathrm{Gal}(L/K)~|~ \chi|_{\mathrm{Gal}(L/L')}=1$} \right\}\right|\\
&=&[L : K]-[L':K]\\
&=&[L : K]\left(1-\frac{1}{[L:L']}\right)\\
&\geq& \displaystyle\frac{1}{2}[L : K].
\end{eqnarray*}
On the other hand, we deduce
\begin{eqnarray*}
|G_2|
&=&\left|\left\{\textrm{characters $\chi$ of $\mathrm{Cl}(\F)~|~ \chi|_{\mathrm{Cl}(K_\F/L)}=1$ and $ \F_{\chi}\,|\,\F\mathfrak{p}^{-e_\mathfrak{p}}$ for some $\mathfrak{p}\in \mathbf{h}_{L,\F}$} \right\}\right|-1\\
&\leq&\sum_{\mathfrak{p}\in \mathbf{h}_{L,\F}}\left|\left\{\textrm{characters $\chi$ of $\mathrm{Cl}(\F\mathfrak{p}^{-e_\mathfrak{p}})~|~ \chi|_{\mathrm{Cl}(K_{\F\mathfrak{p}^{-e_\mathfrak{p}}}/L\cap K_{\F\mathfrak{p}^{-e_\mathfrak{p}}})}=1$} \right\}\right|-1\\
&=&\sum_{\mathfrak{p}\in \mathbf{h}_{L,\F}}[L\cap K_{\F\mathfrak{p}^{-e_\mathfrak{p}}}:K]-1\\
&=&[L : K]\left(\sum_{\mathfrak{p}\in \mathbf{h}_{L,\F}}\frac{1}{[L:L\cap K_{\F\mathfrak{p}^{-e_\mathfrak{p}}}]}\right)-1\\
&\leq& \frac{1}{2}[L : K]-1\quad\textrm{by (\ref{assumption})}.
\end{eqnarray*}
Hence $|G_1|>|G_2|$ and so the claim is proved.
\par
Choose a class $D\in \mathrm{Cl}(K_\F/L')\setminus \mathrm{Cl}(K_\F/L)$ such that $\chi(D)\neq 1$.
We then see from the proof of Lemma \ref{existence of character} that there is a character $\psi$ of $\mathrm{Cl}(\F)$ satisfying $\chi\psi|_{\mathrm{Cl}(K_\F/L)}=1$, $\chi\psi(D)\neq 1$,
 $\mathfrak{f}_{\chi}\, |\, \mathfrak{f}_{\chi\psi}$ and $\mathfrak{p} \,|\, \mathfrak{f}_{\chi\psi}$ for  every prime ideal factor $\mathfrak{p}$ of $\mathfrak{f}$.
We replace $\chi$ by $\chi\psi$.
\par

Since $\chi$ is nontrivial and $\F_\chi\neq \o_K$, we obtain $S_\F(\overline{\chi})\neq 0$ by Proposition \ref{L-function relation}.
On the other hand, we derive that
\begin{eqnarray*}
S_\F(\overline{\chi})&=&\frac{1}{n}\sum_{E\in\mathrm{Cl}(\F)}\overline{\chi}(E)\log\big|g_\F(E)^n\big|\\
&=&\frac{1}{n}\sum_{E\in\mathrm{Cl}(\F)}\overline{\chi}(E)\log\big|(g_\F(C_0)^n)^{\sigma_\F(E)}\big|\quad\textrm{(by Proposition \ref{Galois action})}\\
&=&\frac{1}{n}
\sum_{\substack{E_1\in\mathrm{Cl}(\F)\\ E_1\bmod{\mathrm{Cl}(K_\F/L')}}}
~\sum_{\substack{E_2\in\mathrm{Cl}(K_\F/L')\\ E_2\bmod{\mathrm{Cl}(K_\F/L)}}}
~\sum_{E_3\in\mathrm{Cl}(K_\F/L)}
\overline{\chi}(E_1E_2E_3)\log\big|(g_\F(C_0)^n)^{\sigma_\F(E_1E_2E_3)}\big|\\
&=&\frac{1}{n}\sum_{E_1}\overline{\chi}(E_1)\sum_{E_2}\overline{\chi}(E_2)
\log\big|N_{K_{\F}/L}(g_\F(C_0)^{n})
^{\sigma_\F(E_1)\sigma_\F(E_2)}\big| \quad\textrm{since $\chi|_{\mathrm{Cl}(K_\F/L)}=1$}\\
&=&\frac{1}{n}\sum_{E_1}\overline{\chi}(E_1)\log\big|N_{K_{\F}/L}(g_\F(C_0)^{n})
^{\sigma_\F(E_1)}\big|\Big(\sum_{E_2}\overline{\chi}(E_2)\Big)
\\
&=&0,
\end{eqnarray*}
because $N_{K_{\F}/L}(g_\F(C_0)^{n})\in L'$ and $\chi|_{\mathrm{Cl}(K_\F/L')}\neq 1$.
This is a contradiction, and so $L'=L$.
\par
Since $L'$ is an abelian extension of $K$ and 
\begin{equation*}
N_{K_{\F}/L}(g_\F(C_0)^{n})^{\sigma_\F(C)}=N_{K_{\F}/L}(g_\F(C)^{n})\quad\textrm{for $C\in\mathrm{Cl}(\F)$},
\end{equation*}
we conclude that $L=L'=K(N_{K_{\F}/L}(g_\F(C)^{n}))$ as desired.
\end{proof}

\begin{remark}
If $\mathfrak{f}=\mathfrak{p}^n$ is a power of a prime ideal $\mathfrak{p}$ of $K$, then the assumption $(\ref{assumption})$ is always satisfied since $| \mathbf{h}_{L,\F}|\leq 1$.

\end{remark}

Now, consider the case where $L=K_\F$.
One can readily show that
\begin{equation*}
\mathbf{h}_{K_\F,\F}=\{\textrm{a prime ideal factor $\mathfrak{p}$ of $\F$} ~|~ 
\textrm{$|\mathbf{G}_\mathfrak{p}|=1$ or $2$}\},
\end{equation*}
and hence \cite[Theorem 4.6]{K-Y} is a special case of Theorem \ref{singular value} for $L=K_\F$ as follows:

\begin{corollary}
Let $\mathfrak{f}=\prod_\mathfrak{p}\mathfrak{p}^{e_\mathfrak{p}}$ be  a nonzero proper integral ideal of $K$.
Assume that $K_\F\neq K_{\F\mathfrak{p}^{-e_\mathfrak{p}}}$ for every prime ideal factor $\mathfrak{p}$ of $\mathfrak{f}$, and 
\begin{equation*}
\sum_{\mathfrak{p}\in \mathbf{h}_{K_\F,\F}}\frac{1}{\phi(\mathfrak{p}^{e_\mathfrak{p}})}\leq\frac{1}{2}.
\end{equation*}
Then for any nonzero integer $n$ and $C\in\mathrm{Cl}(\F)$, we have
\begin{equation*}
K_\F=K(g_\F(C)^n).
\end{equation*}

\end{corollary}
\begin{proof}
See \cite[Theorem 4.6]{K-Y}.
\end{proof}

\begin{remark}
One can show without difficulty that $|\mathbf{G}_\mathfrak{p}|=1$ or $2$ if and only if $\mathfrak{p}^{e_\mathfrak{p}}$ satisfies one of the following conditions (\cite[Lemma 4.4]{K-Y}):
\begin{itemize}
\item[] \textbf{Case 1} : $K\neq \mathbb{Q}(\sqrt{-1}), \mathbb{Q}(\sqrt{-3})$ 
\begin{itemize}
\item[$\bullet$] $2$ is not inert in $K$, $\mathfrak{p}$ is lying over $2$ and $e_\mathfrak{p}=1,2$ or $3$.
\item[$\bullet$] $3$ is not inert in $K$, $\mathfrak{p}$ is lying over $3$ and $e_\mathfrak{p}=1$.
\item[$\bullet$] $5$ is not inert in $K$, $\mathfrak{p}$ is lying over $5$ and $e_\mathfrak{p}=1$.
\end{itemize}
\item[] \textbf{Case 2} : $K=\mathbb{Q}(\sqrt{-1})$ 
\begin{itemize}
\item[$\bullet$] $\mathfrak{p}$ is lying over $2$ and $e_\mathfrak{p}=1,2,3$ or $4$.
\item[$\bullet$] $\mathfrak{p}$ is lying over $3$ and $e_\mathfrak{p}=1$.
\item[$\bullet$] $\mathfrak{p}$ is lying over $5$ and $e_\mathfrak{p}=1$.
\end{itemize}
\item[] \textbf{Case 3} : $K=\mathbb{Q}(\sqrt{-3})$ 
\begin{itemize}
\item[$\bullet$] $\mathfrak{p}$ is lying over $2$ and $e_\mathfrak{p}=1$ or $2$.
\item[$\bullet$] $\mathfrak{p}$ is lying over $3$ and $e_\mathfrak{p}=1$ or $2$.
\item[$\bullet$] $\mathfrak{p}$ is lying over $7$ and $e_\mathfrak{p}=1$.
\item[$\bullet$] $\mathfrak{p}$ is lying over $13$ and $e_\mathfrak{p}=1$.
\end{itemize}
\end{itemize}
\end{remark}

\begin{example}\label{example}
Let $K=\mathbb{Q}(\sqrt{-11})$ and $L$ be a finite abelian extension of $K$ such that $K\subsetneq L\subset K_{\F}$ for some nonzero proper integral ideal $\F$ of $K$.
Then $H_K=K$.
\begin{itemize}
\item[(i)] Let $\mathfrak{f}=5\mathcal{O}_K$. 
Then $\F$ is a prime ideal.
Since $|\mathbf{h}_{L,\F}|\leq 1$, we conclude by Theorem \ref{singular value} 
\begin{equation*}
L=K(N_{K_{\F}/L}(g_\F(C)^{n}))
\end{equation*}
for any nonzero integer $n$ and $C\in\mathrm{Cl}(\mathfrak{f})$.
\item[(ii)]
Let $\mathfrak{f}=22\mathcal{O}_K$. 
Then $\mathfrak{f}=\mathfrak{p}_1\mathfrak{p}_2^2$ with prime ideals $\mathfrak{p}_1=2\mathcal{O}_K$ and $\mathfrak{p}_2=\sqrt{-11}\mathcal{O}_K$.
Observe that $|\mathbf{G}_{\mathfrak{p}_1}|=3$, $|\mathbf{G}_{\mathfrak{p}_2}|=55$ and $[K_\F :K]=165$.
Hence $[K_\F :L]=1,2,3,5,11,15,33$ or $55$.
Since
\begin{equation*}
\begin{array}{rcll}
\mathrm{ord}_{3}(|\mathbf{G}_{\mathfrak{p}_1}|)&>&\mathrm{ord}_{3}([K_{\F}:L])&\quad\textrm{if $[K_\F :L]=1,5,11,15,55$},\\
\mathrm{ord}_{5}(|\mathbf{G}_{\mathfrak{p}_2}|)&>&\mathrm{ord}_{5}([K_{\F}:L])&\quad\textrm{if $[K_\F :L]=3,33$},\\
\mathrm{ord}_{11}(|\mathbf{G}_{\mathfrak{p}_2}|)&>&\mathrm{ord}_{11}([K_{\F}:L])&\quad\textrm{if $[K_\F :L]=15$},
\end{array}
\end{equation*}
we get $|\mathbf{h}_{L,\F}|= 0$ or $1$ for any case.
Therefore, it follows from Theorem \ref{singular value} that
\begin{equation*}
L=K(N_{K_{\F}/L}(g_\F(C)^{n}))
\end{equation*}
for any nonzero integer $n$ and $C\in\mathrm{Cl}(\mathfrak{f})$.
\end{itemize}
\end{example}

\bibliographystyle{amsplain}

\begin{thebibliography}{10}








\bibitem{Janusz} G. J. Janusz, \textit{Algebraic Number Fields}, 
2nd ed., Graduate Studies in Mathematics, vol. 7. American Mathematical Society, Providence, 1996.

\bibitem{Jung} H. Y. Jung, J. K. Koo and D. H. Shin, \textit{Ray class invariants over imaginary quadratic fields}, Tohoku Math. J., 63 (2011), 413--426.




\bibitem {K-Y} J. K. Koo and D. S. Yoon, \textit{Construction of ray class fields by smaller
generators and applications}, Proc. Roy. Soc. Edinburgh Sect. A 147 (2017), no. 4, 781--812. 



\bibitem {Kubert} D. Kubert and S. Lang, \textit{Modular Units}, Grundlehren der mathematischen Wissenschaften
244, Spinger-Verlag, New York-Berlin, 1981.




\bibitem {Lang} S. Lang, \textit{Elliptic Functions}, 2nd ed., Spinger-Verlag, New York, 1987.








\bibitem{Ramachandra} K. Ramachandra,  \textit{Some applications of Kronecker's limit formulas}, 
Ann. Math. (2) 80 (1964), 104--148. 



\bibitem {Schertz} R. Schertz, \textit{Complex multiplication}, New Math. Monographs 15, Cambridge University Press, Cambridge, 2010.

\bibitem{Serre} J. -P. Serre, \textit{A Course in Arithmetic}, 
Graduate Texts in Mathematics, 7. Springer-Verlag, New York-Heidelberg, 1973.






\bibitem{Siegel} C. L. Siegel, \textit{Lectures on advanced analytic number theory}, 
Tata Institute of Fundamental Research Lectures on Mathematics, 23, Tata Institute of Fundamental Research, Bombay, 1965. 











\end{thebibliography}

\address{
Ja Kyung Koo\\
Department of Mathematical Sciences \\
KAIST \\
Daejeon 34141 \\
Republic of Korea} {jkkoo@math.kaist.ac.kr}
\address{
Dong Sung Yoon\\
Department of Mathematical Sciences \\
Ulsan National Institute of Science and Technology \\
Ulsan 44919 \\
Republic of Korea} {dsyoon@unist.ac.kr}

\end{document}